\documentclass[10pt,reqno]{amsart}

\usepackage{a4wide}
\usepackage{dsfont}            
\usepackage{amssymb,amsmath,amsfonts}
\usepackage{mathrsfs}
\usepackage{graphicx} 
\usepackage{xcolor}          
\usepackage[latin10]{inputenc} 
\usepackage[all]{xy}

\newtheorem{proposition}{Proposition}[section]
  \newtheorem{theorem}[proposition]{Theorem}
  
  \newtheorem{lemma}[proposition]{Lemma}
\theoremstyle{definition}
  \newtheorem{definition}[proposition]{Definition}
  \newtheorem{remark}[proposition]{Remark}

\allowdisplaybreaks
\newcommand{\cst}{\ifmmode\mathrm{C}^*\else{$\mathrm{C}^*$}\fi}
\newcommand{\st}{\;\vline\;}

\newcommand{\CC}{\mathbb{C}}
\newcommand{\RR}{\mathbb{R}}

\newcommand{\GG}{\mathbb{G}}

\newcommand{\vtens}{\,\bar{\otimes}\,}
\newcommand{\id}{\mathrm{id}}

\newcommand{\I}{\mathds{1}}

\newcommand{\HH}{\mathbb{H}}
\newcommand{\sM}{\mathsf{M}}
\newcommand{\sN}{\mathsf{N}}

\newcommand{\sS}{\mathsf{S}}
\newcommand{\sB}{\mathsf{B}}

\newcommand{\sA}{\mathsf{A}}

\newcommand{\sH}{\mathsf{H}}

\newcommand{\hh}[1]{\widehat{#1}}

\newcommand{\flip}{\boldsymbol{\sigma}}
\newcommand{\ww}{\mathrm{W}}
\newcommand{\WW}{{\mathds{V}\!\!\text{\reflectbox{$\mathds{V}$}}}}
\newcommand{\Ww}{\mathds{W}}
\newcommand{\wW}{\text{\reflectbox{$\Ww$}}\:\!} 

\newcommand{\U}{\mathrm{U}}
\newcommand{\Uu}{\mathds{U}}

\newcommand{\starr}{\,\underline{*}\,}
\newcommand{\staru}{\,\overline{*}\,}

\DeclareMathOperator{\C}{C}
\DeclareMathOperator{\B}{B}
\DeclareMathOperator{\vN}{vN}
\DeclareMathOperator{\Mor}{Mor}
\DeclareMathOperator{\M}{M}
\DeclareMathOperator{\N}{N}
\DeclareMathOperator{\Dom}{Dom}
\DeclareMathOperator{\Aut}{Aut}

\DeclareMathOperator{\Linf}{\mathnormal{L}^\infty\;\!\!}
\DeclareMathOperator{\Ltwo}{\mathnormal{L}^2\;\!\!}

\numberwithin{equation}{section}

\def\labelitemi{$\blacktriangleright$}
\author{Ramin Faal}
\address{Department of Pure Mathematics, Ferdowsi University of Mashhad, Iran}
\email{faal.ramin@yahoo.com}
\author{Pawe{\l} Kasprzak}
\address{Department of Mathematical Methods in Physics, Faculty of Physics, University of Warsaw, Poland}
\email{pawel.kasprzak@fuw.edu.pl}

\title{Group-like projections for locally compact quantum groups}


\begin{document}

\begin{abstract}Let $\GG$ be a locally compact quantum group. 
We give a 1-1 correspondence between group-like projections in $\Linf(\GG)$  preserved by the scaling group    and idempotent states on the dual quantum group $\hh\GG$. As a byproduct we give a simple proof that normal integrable coideals in $\Linf(\GG)$ which are preserved by the scaling group are in 1-1 correspondence with compact quantum subgroups of $\GG$. 
\end{abstract}

\maketitle


\newlength{\sw}
\settowidth{\sw}{$\scriptstyle\sigma-\text{\rm{weak closure}}$}
\newlength{\nc}
\settowidth{\nc}{$\scriptstyle\text{\rm{norm closure}}$}
\newlength{\ssw}
\settowidth{\ssw}{$\scriptscriptstyle\sigma-\text{\rm{weak closure}}$}
\newlength{\snc}
\settowidth{\snc}{$\scriptscriptstyle\text{\rm{norm closure}}$}
\renewcommand{\labelitemi}{$\bullet$}
\section{Introduction}

 Let $G$ be a group, $X$  a non-empty subset of  $G$ and   $\I_X:G\to \{0,1\}$ its characteristic function. It is easy to check that  $X$ is a subgroup of $G$ if and only if 
\begin{equation}\label{glpemts0}\I_X(st)\I_X(t) = \I_X(s )\I_X(t)\end{equation} for all $s,t\in G$. 
Let $G$ be  a locally compact group    and  $\Delta: \Linf(G)\to \Linf(G)\vtens\Linf(G)$  the comultiplication on $\Linf(G)$: \[\Delta(f)(s,t) = f(st)\] for all $f\in\Linf(G)$. Suppose that  $P\in\Linf(G)$ is a non-zero group-like projection, i.e. $P$ satisfies   \begin{equation}\label{glpemts1}\Delta(P)(\I_G\otimes P) = P\otimes P.\end{equation}   
Equation \eqref{glpemts1} implies that $P$ is a continuous function on $G$ (see Lemma \ref{lemmamult}). Denoting   
\[X = \{s\in G: P(s) = 1\},\] 
  we have $P = \I_X$ and $\I_X$ satisfies \eqref{glpemts0}. In particular  $X$ is a subgroup of $G$ and the continuity of $\I_X$ implies that $X$ is open. Thus we get a 1-1 correspondence between open subgroups of $G$ and group-like projections in $\Linf(G)$. 

Let $G$ be a locally compact group. The Banach dual $ \C_0(G)^*$    of $\C_0(G)$ equipped  with the convolution product is a Banach algebra.  We say that a state $\omega \in\C_0(G)^*$ is an idempotent state on  $\C_0(G)$ if  $\omega*\omega = \omega$. In fact, as proved by Kelley \cite[Theorem 3.4]{Kelley}, there is a 1-1 correspondence between  idempotent states on $\C_0(G)$ and   compact  subgroups of $G$, where given a compact subgroup $H\subset G$  the corresponding state is of the form  $\omega(f) = \int_H f(h)dh$ for all $f\in \C_0(G)$. 

Let $G\ni g\mapsto R_g\in B(\Ltwo(G))$ be the right regular representation, $\vN(G)=\{R_g:g\in G\}''$    the group von Neumann algebra of $G$ and  $\hh\Delta:\vN(G)\to \vN(G)\vtens \vN(G)$  the comultiplication, where $\hh\Delta(R_g) = R_g\otimes R_g$ for all $g\in G$. It is not difficult to see that   $P = \int_H R_h dh\in \vN(G)$ is a group-like projection in $\vN(G)$, i.e. 
 it satisfies 
\[\hat\Delta(P)(\I\otimes P) = P\otimes P.\]
 Theorem \ref{glptoidst} and Kelley's result show  that all group-like projections in  $\vN(G)$ are of this form. In other words we have  a 1-1 correspondence between idempotent states on  $ \C_0(G)$ and group-like projections in $\vN(G)$.  Theorem \ref{lemma0} together with Theorem \ref{glptoidst} yield a    generalization of  this  correspondence to the context of  locally compact quantum groups. 

A locally compact quantum group  $\GG$ is a virtual object that is assigned with a von Neumann algebra    $\Linf(\GG)$   equipped with a comultiplication $\Delta:\Linf(\GG)\to \Linf(\GG)\vtens \Linf(\GG)$. A projection  $P\in\Linf(\GG)$ is called a  group-like projection if 
\[\Delta(P)(\I\otimes P) = P\otimes P.\]
A locally compact quantum group $\GG$ is also assigned with the $\C^*$-algebra  $\C_0(\GG)$  and universal $\C^*$-algebra $\C^u_0(\GG)$. Both Banach duals $\C^u_0(\GG)^*$ and $\C_0(\GG)^*$ are in fact   Banach algebras. We say that a state $\omega\in\C_0^u(\GG)^*$ is an idempotent state (on $\GG$) if $\omega*\omega = \omega$. 
As already mentioned,   our  results establish a 1-1 correspondence between idempotent states on $\GG$ and group-like projections  on the dual   $\hh\GG$ which are preserved by the scaling group of $\hh\GG$. 
As a byproduct of our study   we  get a relatively simple  proof that normal integrable coideals in $\Linf(\GG)$ which are preserved by the scaling group are in 1-1 correspondence with compact quantum subgroups of $\GG$. Our proof,  unlike the previous proof  \cite[Theorem 5.15]{coid_sub_st}, uses only the von Neumann techniques and does not invoke  the universal $\C^*$ - algebra $\C_0^u(\GG)$. 

\section{Preliminaries}
We will denote the minimal  tensor product of $\C^*$-algebras with the symbol $\otimes$. The ultraweak tensor product of von Neumann algebras  will be denoted by $\vtens$. For a $\C^*$-subalgebra $\sB$ of a $\C^*$-algebra the multipliers  $\M(\sA)$ of $\sA$, the norm closed linear span of the set $\big\{ba \st b\in\sB, a\in\sA \big\}$ will be denoted by $\sB\sA$. A morphism  between two $\C^*$-algebras $\sA$ and $\sB$ is a $*$-homomorphism $\pi$ from  $\sA$ into the multiplier algebra $\M(\sB)$,  which is  non-degenerate, i.e $\pi(\sA)\sB = \sB$. We will denote the set of all morphisms from $\sA$ to $\sB$ by $\Mor(\sA,\sB)$. The non-degeneracy of a morphism $\pi$ yields its  natural extension to the unital $*$-homomorphism $\M(\sA)\to \M(\sB)$ also denoted by $\pi$. Let $\sB$ be a $\C^*$-subalgebra of $\M(A)$. We say that $\sB$ is non-degenerate if $\sB\sA = \sA$.  In this case $\M(\sB)$ can  be identified with a $\C^*$-subalgebra of $\M(\sA)$.  The symbol $\flip$ will denote the flip morphism between tensor product of operator algebras. If $X\subset \sA$, where $\sA$ is a $\C^*$-algebra then  $ X ^{\textrm{\tiny{norm-cls}}}$ denotes the norm closure of the  linear span of $X$; if   $X\subset \sM$, where $\sM$ is a von Neumann algebra then $ X ^{ \textrm{\tiny{$\sigma$-weak cls}}}$ denotes  the $\sigma$-weak closure of the linear span of $X$.    For a $\C^*$-algebra $\sA$, the space of all functionals on $\sA$ and the state space of $\sA$ will be denoted by $\sA^*$ and $S(\sA)$ respectively. The predual of a von Neumann algebra $\sN$ will be denoted by $\sN_*$. For a Hilbert space $H$ the $\C^*$-algebras of compact   operators on $H$ will be denoted by $\mathcal{K}(H)$. The algebra of bounded operators acting on $H$ will be denoted by  $B(H)$. For $\xi, \eta\in H$, the symbol $\omega_{\xi, \eta}\in B(H)_*$ is the functional $T\mapsto \langle \xi,T\eta\rangle$.  

 For  the theory of locally compact quantum groups we refer to \cite{univ,KV,KVvN}. Let us recall that a von Neumann algebraic locally compact quantum group is a quadruple $\mathbb{G} = (\Linf(\GG), \Delta,\varphi,\psi)$, where $\Linf(\GG)$ is a von Neumann algebra with a coassociative comultiplication $\Delta\colon\Linf(\GG)\to\Linf(\GG)\vtens\Linf(\GG)$,  $\varphi$ and $\psi$ are, respectively,  normal semifinite  faithful left and right Haar weights on $\Linf(\GG)$.  The GNS Hilbert space of the right Haar weight $\psi_\GG$ will be denoted by  $\Ltwo(\GG)$ and the corresponding GNS map will be denoted by $\eta_\GG$. Let us recall that $\eta_\GG:\mathcal{N}_\psi\to \Ltwo(\GG)$, where $\mathcal{N}_\psi=\{x\in\Linf(\GG):\psi(x^*x)<\infty\}$.  The \emph{antipode}, the \emph{scaling group} and the \emph{unitary antipode} will be denoted by $S$, $(\tau_t)_{t\in \RR}$ and $R$. We have $S = R\circ\tau_{-\frac{i}{2}}$. Moreover, for all $a,b\in\mathcal{N}_\varphi$ the following holds (see \cite[Corollary 5.35]{KV})
\begin{equation}\label{stronleftinv} S((\id\otimes\varphi)(\Delta(a^*)(\I\otimes b))) = (\id\otimes\varphi)((\I\otimes a^*)\Delta(b)).\end{equation}
We will denote $(\sigma^\varphi_t)_{t\in\RR}$ and $(\sigma^{\psi}_t)_{t\in\RR}$ the \emph{modular automorphism groups} assigned to $\varphi$  and $\psi$ respectively. 

The multiplicative unitary $\ww^\GG\in B(\Ltwo(\GG)\otimes \Ltwo(\GG))$  is a unique unitary operator such that 
\[\ww^\GG(\eta_\GG(x)\otimes\eta_\GG(y)) = (\eta_\GG \otimes\eta_\GG )(\Delta_\GG(x)(\I\otimes y))\] 
for all $x,y\in D(\eta_\GG)$;  
 $\ww^\GG$  satisfies the pentagonal equation $\ww^\GG_{12}\ww^\GG_{13}\ww^\GG_{23} = \ww^\GG_{23}\ww^\GG_{12}$ \cite{BS,mu}. Using $\ww^\GG$,  $\GG$ can be recovered as follows:
\[\begin{split}
\Linf(\GG) =&\bigl\{ (\omega\otimes\id)\ww^\GG\st\omega\in B(\Ltwo(\GG))_*\bigr\} ^{\textrm{\tiny{$\sigma$-weak cls}}},\\
\Delta_\GG(x) =&\ww^\GG(x\otimes\I){\ww^\GG}^*.
\end{split}\] 
A locally compact quantum group admits a dual object $\hh\GG$. It can be described in terms of  
 $\ww^{\hh\GG}=\flip({\ww^{\GG}})^*$
\[\begin{split} 
\Linf(\hh\GG)&=\bigl\{( \omega\otimes\id)\ww^{\hh\GG}\st\omega\in B(\Ltwo(\GG))_*\bigr\}^{\textrm{\tiny{$\sigma$-weak cls}}},\\ \Delta_{\hh\GG}(x)&=\ww^{\hh\GG}(x\otimes\I){\ww^{\hh\GG}}^*.
 \end{split}\] Note that $\ww^\GG\in\Linf(\hh\GG)\bar\otimes\Linf(\GG)$. 
\begin{definition}\label{von_neumann}
A von Neumann subalgebra $\sN$ of $\Linf(\GG)$ is called 
\begin{itemize}
\item \emph{Left coideal}  if $\Delta_\GG(\sN)\subset\Linf(\GG)\vtens\sN$;
\item \emph{Invariant subalgebra} if $\Delta_\GG(\sN)\subset\sN\vtens\sN$;
\item \emph{Baaj-Vaes subalgebra} if $\sN$ is an invariant subalgebra of $\Linf(\GG)$ which is preserved by the unitary antipode $R$  and the scaling group $(\tau_t)_{t\in\RR} $ of $\GG$;
\item \emph{Normal} if $\ww^\GG(\I\otimes\sN){\ww^\GG}^*\subset\Linf(\hh\GG)\vtens\sN$;
\item \emph{Integrable} if the set of integrable elements with respect to the right Haar weight $\psi_\GG$ is dense in $\sN^+$; in other words, the restriction of $\psi_\GG$ to $\sN$ is semifinite.
\end{itemize}
\end{definition}

If $\sN$ is a coideal of $\Linf(\GG)$, then $\tilde\sN=\N^\prime\cap\Linf(\hh\GG)$ is a   coideal of $\Linf(\hh\GG)$ called the \emph{codual} of $\N$; it turns out that  $\tilde{\tilde\sN}=\sN$ (see \cite[Theorem 3.9]{embed}). 

The $\C^*$-algebraic version $(\C_0(\GG),\Delta_\GG)$ of a given quantum group $\GG$  is recovered from $\ww^\GG$ as follows 
\[\begin{split}\C_0(\GG) &=\bigl\{(\omega\otimes\id)\ww^\GG\st \omega\in B(\Ltwo(\GG))_*\bigr\}^{\textrm{\tiny{norm-cls}}},\\
\Delta_\GG(x) &=\ww^\GG(x\otimes\I){\ww^\GG}^*.
\end{split}\]
The comultiplication can be viewed as a morphism $\Delta_\GG\in \Mor(\C_0(\GG),\C_0(\GG)\otimes\C_0(\GG))$ and we have  $\ww^\GG\in\M(\C_0(\hh\GG)\otimes \C_0(\GG))$. Since $\M(\C_0(\hh\GG)\otimes \C_0(\GG))\subset \M(\mathcal{K}(\Ltwo(\GG))\otimes \C_0(\GG))$ we conclude that for all $x\in \Linf(\GG)$ 
\begin{equation}\Delta_\GG(x) = \ww^\GG(x\otimes\I){\ww^\GG}^*\in\M(\mathcal{K}(\Ltwo(\GG))\otimes \C_0(\GG)).\end{equation}
Replacing $\Delta_\GG$ with $\Delta_{\GG^{\textrm{op}}}$ we also get that 
\begin{equation}\label{cont2} \Delta_\GG(x) \in \M(\C_0(\GG)\otimes \mathcal{K}(\Ltwo(\GG))) \end{equation} for all $x\in\Linf(\GG)$. 
 
Let $\sH$ be a Hilbert space and $\U\in\M(\C_0(\GG)\otimes\mathcal{K}(\sH))$ a unitary. We say that $\U$ is a representation of $\GG$ on $\sH$ if \[(\Delta_\GG\otimes\id)(\U) = \U_{13}\U_{23}.\] 
Let us recall the definition of an action  of a quantum group  $\GG$ on a von Neumann algebra.
\begin{definition}\label{def:action_v_C} A \emph{(left) action} of quantum group $\GG$ on a 
 von Neumann algebra $\sN$ is a unital injective normal $*$-homomorphism  $\alpha:\sN\to\Linf(\GG)\vtens\sN$ s.t. $(\Delta_\GG\otimes\id)\circ\alpha=(\id\otimes\alpha)\circ\alpha$. If $\sM\subset \sN$ is a von Neumann subalgebra then we say that $\sM$ is preserved by $\alpha$ if $\alpha(\sM)\subset \Linf(\GG)\vtens\sM$. 
\end{definition} 
Given an action $\alpha:\sN\to\Linf(\GG)\vtens\sN$ we have (see \cite[Corollary 2.6]{embed})   \[\sN = \{(\mu\otimes\id)\alpha(x):x\in\sN,\mu\in\Linf(\GG)_*\}^{ \textrm{\tiny{$\sigma$-weak cls}}}\] which will be referred to as {\it the Podle\'s condition}. 
We can always find  a unitary representation $\U\in\M(\C_0(\GG)\otimes\mathcal{K}(\sH))$ on a Hilbert space $\sH$ and a normal faithful $*$-homomorphism $\pi:\sN\to \B(\sH)$ such that 
\[(\id\otimes\pi)(\alpha(x)) =\U^*(\I\otimes\pi(x) )\U.\] In this case we shall say that $\U$ implements the action $\alpha$. For the construction of the canonical implementation see \cite{VaesCan}.

A locally compact quantum group $\GG$ is assigned with a universal version \cite{univ}. The universal version $\C_0^u(\GG)$  of $\C_0(\GG)$    is equipped  with a comultiplication $\Delta_\GG^u \in\Mor(\C_0^u(\GG),\C_0^u(\GG)\otimes \C_0^u(\GG))$. 
The \emph{counit} is a $*$-homomorphism $\varepsilon: \C_0^u(\GG)\to\CC$ satisfying $(\id\otimes\varepsilon)\circ\Delta_\GG^u=\id=(\varepsilon\otimes\id)\circ\Delta_\GG^u$. Multiplicative unitary $\ww^\GG\in\M(\C_0(\hh\GG)\otimes\C_0(\GG))$ admits the universal lift  $\WW^\GG\in\M(\C_0^u(\hh\GG)\otimes \C_0^u(\GG))$.  The reducing morphisms for $\GG$ and $\hh\GG$ will be denoted by $\Lambda_\GG\in\Mor(\C_0^u(\GG),\C_0(\GG))$ and $\Lambda_{\hh\GG}\in\Mor(\C_0^u(\hh\GG),\C_0(\hh\GG))$  respectively. We have
$(\Lambda_{\hh\GG}\otimes\Lambda_\GG)(\WW^\GG)=\ww^\GG$. 
We shall also use the half-lifted versions of $\ww^\GG$, $\Ww^\GG=(\id\otimes\Lambda_\GG)(\WW^\GG)\in\M(\C_0^u(\hh\GG)\otimes \C_0(\GG))$ and $\wW^\GG=(\Lambda_{\hh\GG}\otimes\id)(\WW^\GG )\in\M(\C_0(\hh\GG)\otimes \C_0^u(\GG))$. They satisfy the appropriate versions of pentagonal equation 
\[\begin{split}
\Ww^\GG_{12}\Ww^\GG_{13}\ww^\GG_{23}&=\ww^\GG_{23}\Ww^\GG_{12},\\
\ww^\GG_{12}\wW^\GG_{13}\wW^\GG_{23}&=\wW^\GG_{23}\ww^\GG_{12}.
\end{split}\]
The half-lifted versions of comultiplication  is denoted by  $\Delta_r^{r,u}\in\Mor(\C_0(\GG),\C_0(\GG)\otimes \C_0^u(\GG))$ 
and 
$\hh{\Delta}_r^{r,u}\in\Mor(\C_0(\hh\GG),\C_0(\hh\GG)\otimes \C_0^u(\hh\GG))$, e.g.  
\[\begin{split}
\Delta_r^{r,u}(x)&=\wW^\GG(x\otimes\I){\wW^\GG}^*,~~~~~x\in\C_0(\GG).
\end{split} \]
  We have 
\begin{equation}\label{LDDrL}
\begin{split}
(\Lambda_\GG\otimes\id)\circ\Delta_\GG^u &= \Delta_r^{r,u}\circ\Lambda_\GG,\\
(\Lambda_{\hh\GG}\otimes\id)\circ\Delta_{\hh\GG}^u &= \hh{\Delta}_r^{{r,u}} \circ\Lambda_{\hh\GG}.
\end{split}
\end{equation}
If $\U\in\M(\C_0(\GG)\otimes\mathcal{K}(\sH))$ is a unitary representation of $\GG$ on a Hilbert space then there exists a unique unitary $\Uu\in\M(\C^u_0(\GG)\otimes\mathcal{K}(\sH))$ such that $\U = (\Lambda_\GG\otimes\id)(\Uu)$ and   \[(\Delta^u\otimes\id)(\Uu) = \Uu_{13}\Uu_{23}.\]
Actually $ \Uu_{23} = \U^*_{13}( \Delta_r^{r,u}\otimes\id)(\U)$. 

 Given a locally compact quantum group $\GG$, the comultiplications $\Delta_\GG$ and $\Delta_\GG^u$ induce Banach algebra structures on $\Linf(\GG)_*$ and $\C_0^u(\GG)^*$ respectively.  The corresponding multiplications will be denoted by $\starr$ and $\staru$. We shall identify $\Linf(\GG)_*$ with a subspace of $\C_0^u(\GG)^*$ when convenient. Under this identification $\Linf(\GG)_*$ forms a two sided ideal in $\C_0^u(\GG)^*$.   Following \cite{univ},  for any $\mu\in\C_0^u(\GG)^*$ we   define a normal map $ \Linf(\GG)\to \Linf(\GG)$ such that $ x\mapsto(\id\otimes\mu)(\wW^\GG(x\otimes\I){\wW^\GG}^*)$ for all $x\in\Linf(\GG)$. We shall use a notation $\mu\staru x=  (\id\otimes\mu)(\wW^\GG(x\otimes\I){\wW^\GG}^*)$.

\begin{theorem}\label{thm_gen_el}
Let $\sN$ be a von Neumann algebra and $\alpha:\sN \to \Linf(\GG)\vtens\sN$  an action of $\GG$ on $\sN$. Let $x\in \sN$, $x^* = x$ and 
\[\sN_x=\{(\mu\otimes \id)(\alpha(x)):\mu\in\Linf(\GG)_*\}''.\] Then $\sN_x$ is the smallest unital von Neumann  subalgebra of  $\sN$ preserved by $\GG$ and containing $x$.
\end{theorem}
\begin{proof}
Let us consider 
\[\sS =\{(\mu\otimes \id)(\alpha(x)):\mu\in\Linf(\GG)_*\}.\]
Then $\sS$ forms a selfadjoint subset of $\sS$. In particular $\sN_x$ is (unital) von Neumann algebra generated by $\sS$. 
Noting that 
\[
    (\omega_1\otimes\id)(\alpha((\omega_2\otimes \id)\alpha(x))) = (\omega_2\starr\omega_1\otimes\id)(\alpha(x))\in\sN_x
\] we conclude that $\sN_x$ is preserved by $\GG$.

Every   $\sM\subset \sN$ preserved by $\GG$ and containing $x$ must contain $\sN_x$, so it remains to prove that $x\in\sN_x$. For this we may assume that $\sN\subset \B(\sH)$ and $\alpha$ is implemented by a unitary  representation $\U\in\M(\C_0(\GG)\otimes \mathcal{K}(\sH))$
\[\alpha(x) = \U^*(\I\otimes x)\U.\] 
Unitary implementation enables us to  define a   morphism $\alpha_0\in\Mor(\mathcal{K}(\sH),\C^u_0(\GG)\otimes \mathcal{K}(\sH))$, where $\alpha_0(x) = \U^*(\I\otimes x)\U$. 
Thus, using natural extension of the morphism $\alpha_0$ to  $\B(\sH) =\M(\mathcal{K}(\sH))$  we can further extend $\alpha$ to an action  on $\B(\sH)$ and 
  we shall assume in what follows that $\sN = \B(\sH)$.
As the conclusion of the above observation we see that,  given a $\C^*$-algebra $\sB$, an element $X\in\M(\sB\otimes \mathcal{K}(\sH))$ and a functional $\mu\in\sB^*$ we have \begin{equation}\label{cont0}
\alpha((\mu\otimes \id)(X)) = (\mu\otimes\id\otimes\id)((\id\otimes\alpha)(X)).
\end{equation}

Let $\Uu\in\M(\C^u_0(\GG)\otimes \mathcal{K}(H))$ be the universal lift of $\U$. Let us note that 
\[\sM := \{(\mu\otimes \id)(\Uu^*(\I\otimes x)\Uu):\mu\in\C_0^u(\GG)^*\}''\]
is a von Neumann subalgebra of $\B(\sH)$ containing $x$ (for the latter take $\mu = \varepsilon$) and $\sN_x\subset \sM$.  Furthermore, for every $\omega\in\Linf(\GG)_*$ we have 
\begin{equation}\label{cont}
    (\omega\otimes\id)(\alpha((\mu\otimes \id)(\Uu^*(\I\otimes x)\Uu)) = (\mu\staru\omega\otimes\id)(\alpha(x))\in\sN_x\subset\sM
\end{equation} where we used \eqref{cont0}. This shows that $\sM$ is preserved by $\GG$ (note that for the proof of the containment "$\in$" in Eq. \eqref{cont} we use $\mu\staru\omega\in\Linf(\GG)_*$). Since the action of $\GG$ on $\sM$ satisfies the Podle\'s condition, $\sM$   is generated by elements of the form $(\mu\staru\omega\otimes\id)(\alpha(x))$, $\mu\in\C_0^u(\GG)^*$, $\omega\in\Linf(\GG)_*$. Since $\mu\staru\omega\in\Linf(\GG)_*$, we conclude that $\sM\subset \sN_x$ and in particular $x\in\sN_x$. 
\end{proof}
\begin{remark}
If in the context of Theorem \ref{thm_gen_el} we start with a not necessary self-adjoint  $x\in\sN$, then the smallest von Neumann subalgebra of $\sN$ containing $x$ is given by \[\sN_x=\{(\mu\otimes \id)(\alpha(x)), (\mu\otimes \id)(\alpha(x^*)):\mu\in\Linf(\GG)_*\}''.\]
\end{remark}
\begin{definition}\label{Ggen}
Let $\sN$ be a von Neumann algebra with an action $\alpha:\sN\to\Linf(\GG)\vtens\sN$ of a locally compact quantum group $\GG$ and let $x\in\sN$. We say that $\sN$ is $\GG$-generated by  $x$  if $\sN_x = \sN$.  
\end{definition}
A state $\omega\in S(\C_0^u(\GG))$ is said to be an {\it idempotent state} if $\omega\staru\omega=\omega$. For a nice survey   describing the history and motivation behind the study of idempotent states see \cite{Salmi_Survey}.  For the theory of idempotent state we refer   to \cite{SaS}.  We shall use \cite[Proposition 4]{SaS} which in particular states that an  idempotent state  $\omega\in S(\C_0^u(\GG))$ is preserved by the universal scaling group $\tau_t^u$ and the universal unitary antipode $R^u:\C_0^u(\GG)\to\C_0^u(\GG)$, i.e.
\begin{equation}\label{prescgr}
\omega\circ\tau_t^u = \omega=\omega\circ R^u 
\end{equation}
for all $t\in\mathbb{R}$. An idempotent  state $\omega\in S(\C_0^u(\GG))$ yields a conditional expectation $E_\omega:\C_0(\GG)\to\C_0(\GG),$ (see \cite{SaS})
\[E_\omega(x) =\omega\staru x\] for all $x\in\C_0(\GG)$. Using \eqref{prescgr}, we easily get 
\begin{equation}\label{prescgr1}\tau_t(E_\omega(x)) = E_\omega(\tau_t(x)).\end{equation}
Conditional expectation extends to $E_\omega:\Linf(\GG)\to\Linf(\GG)$ and clearly  \eqref{prescgr1} holds for all $x\in\Linf(\GG)$. The image $\sN = E_\omega(\Linf(\GG))$ of $E_\omega$ forms a coideal in $\Linf(\GG)$.

Let $\HH$ and $\GG$ be  locally compact quantum groups. A morphism $\pi\in\Mor(\C_0^u(\GG),\C_0^u(\HH))$ such that 
\[
(\pi\otimes\pi)\circ\Delta_\GG^u=\Delta_\HH^u\circ\pi
\] is said to define a homomorphism from $\HH$ to $\GG$. If $\pi(\C_0^u(\GG))=\C_0^u(\HH)$, then $\HH$ is called  \emph{Woronowicz - closed  quantum subgroup} of $\GG$ \cite{DKSS}. A homomorphism from $\HH$ to $\GG$ admits the dual homomorphism  ${\hh\pi}\in\Mor(\C_0^u(\hh\HH), \C_0^u(\hh\GG))$ such that 
\[
(\id\otimes\pi)(\WW^\GG) = ({\hh\pi}\otimes\id)(\WW^\HH).
\]
A homomorphism  from  $\HH$ to $\GG$ identifies $\HH$ as a  \emph{closed quantum subgroup} of $\GG$ if there exists an injective  normal unital $*$-homomorphism $\gamma:\Linf(\hh\HH)\to\Linf(\hh\GG)$ such that 
\[
\Lambda_{\hh\GG}\circ{\hh\pi}(x) = \gamma\circ\Lambda_{\hh\HH}(x)
\] for all $x\in\C_0^u(\hh\HH)$. 
Let $\HH$ be a closed quantum subgroup of $\GG$, then $\HH$ acts on $\Linf(\GG)$   (in the von Neumann algebraic sense) by the following formula 
\[
 \alpha:\Linf(\GG)\to\Linf(\GG)\vtens\Linf(\HH),~~~ x\mapsto V(x\otimes \I)V^*,
 \]
where \begin{equation}\label{eq:def_V}V =(\gamma\otimes\id)(\ww^\HH).\end{equation} The fixed point space of $\alpha$  is denoted by  \[\Linf(\GG/\HH) =\big\{x\in\Linf(\GG)\st \alpha(x)=x\otimes\I \big\}\]  and referred to as the algebra of bounded functions on the quantum homogeneous space $\GG/\HH$. If $\HH$ is  a compact quantum subgroup of $\GG$, then there is a conditional expectation $E:\Linf(\GG)\to\Linf(\GG)$ onto  $\Linf(\GG/\HH)$ which is defined by 
\begin{equation}\label{eq:def_ce_H} E=(\id\otimes\psi_\HH)\circ\alpha,\end{equation} where $\psi_\HH$ is the Haar state of $\HH$. 

According to [9, Definition 2.2] we
say that $\HH$ is an  \emph{open quantum subgroup} of $\GG$ if there is a surjective normal $*$-homomorphism
$\rho:\Linf(\GG)\to\Linf(\HH)$  such that 
\[\Delta_\HH \circ\rho= (\rho\otimes\rho)\circ\Delta_\GG.
\]
Every open quantum subgroup is closed \cite[Theorem 3.6]{KKS}. We recall that a projection $P\in\Linf(\GG)$ is a  \emph{group-like projection} if $\Delta_\GG(P)(\I\otimes P)=P\otimes P$. Note that \eqref{cont2} implies that $\Delta_\GG(P)(\I\otimes P)\in\M(\C_0(\GG)\otimes \mathcal{K}(\Ltwo(\GG))$. In particular we have  
\begin{lemma}\label{lemmamult}
Let $P\in\Linf(\GG)$ be a group-like projection. Then $P\in\M(\C_0(\GG))$. 
\end{lemma}
 There is a 1-1 correspondence between (isomorphism  classes of) open quantum subgroups of $\GG$ and  central group-like projections in $\GG$ \cite[Theorem 4.3]{KKS}. The group-like projection assigned to $\HH$, i.e.  the central support of $\rho$, will be denoted by $\I_\HH$. 

\section{From idempotent states to  group-like projections}
Let $\GG$ be a locally compact quantum group and  $\omega\in\C_0^u(\GG)^*$   an idempotent state on $\GG$ and let $E_\omega:\Linf(\GG)\to\Linf(\GG)$ be the conditional expectation assigned to $\omega$:
\[E_\omega(x) = \omega\staru x.\] We note that 
\begin{align*}\eta_\GG(E_\omega(x)) &= \eta_\GG(\omega\staru x)\\ &=(\id\otimes\omega)(\wW)\eta_\GG(x),
\end{align*} where in the last equality we use \cite[Proposition 7.4]{int}. 
The element  $(\id\otimes\omega)(\wW)\in\Linf(\hh\GG)$ is a hermitian projection which we denote by $P_\omega$. In particular 
\begin{equation}\label{Pom}\eta_\GG(E_\omega(x)) = P_\omega\eta_\GG(x).\end{equation}
 Let $\sN = E_\omega(\Linf(\GG))$ be the coideal assigned to $\omega$. The set
\begin{equation}\label{dense}E_\omega(\{( \mu\otimes\id)(\ww):\mu\in\Linf(\hh\GG)_*\})=\{(P_\omega\cdot\mu\otimes\id)(\ww):\mu\in\Linf(\hh\GG)_*\}\end{equation} is weakly dense in $\sN$. 

Let us recall that $\tilde{\sN}\subset\Linf(\hh\GG)$ denotes the codual coideal of $\sN$. Since $\sN$ is preserved by $\tau^\GG$, $\tilde\sN$ is preserved by $\tau^{\hh\GG}$.
\begin{theorem}\label{lemma0}
Adopting the above notation we have 
\[\sN = \{x\in\Linf(\GG):P_\omega x = x P_\omega\}\] and 
\[\tilde{\sN} = \{y\in\Linf(\hh\GG):\Delta_{\hh\GG}(y)(\I\otimes P_\omega) = y\otimes P_\omega\}.\]
Moreover, $P_\omega \in\tilde{\sN}$ is a minimal central projection of $\tilde{\sN}$ and it satisfies 
\begin{itemize}
\item $\tau^{\hh\GG}_t(P_\omega) = P_\omega$ for all $t\in\mathbb{R}$;
\item $R^{\hh\GG}(P_\omega) = P_{\omega}$;
\item $\sigma^{\hh\psi}_t(P_\omega) = P_\omega$ for all $t\in\mathbb{R}$;
\item $\sigma^{\hh\varphi}_t(P_\omega) = P_\omega$ for all $t\in\mathbb{R}$;
\item $\Delta_{\hh\GG}(P_\omega)(\I\otimes P_\omega) = P_\omega\otimes P_\omega=\Delta_{\hh\GG}(P_\omega)(P_\omega\otimes \I)$.
\end{itemize}
\end{theorem}
\begin{proof}
The equalities   $\tau^{\hh\GG}_t(P_\omega) = P_\omega$ and $R^{\hh\GG}(P_\omega) = P_{\omega}$   follow  easily from \eqref{prescgr}.

Let $x\in\Linf(\GG)$. 
Using \eqref{Pom} we  see that the condition 
\begin{equation}\label{cond1}P_\omega x = x P_\omega\end{equation} holds if and only if 
\[\eta_\GG(E_\omega(xz)) = \eta_\GG(x E_\omega(z))\] for all $z\in\mathcal{N}_\psi$. The latter  is equivalent to the identity $E_\omega(xz) = xE_\omega(z)$ holding for all $z\in\mathcal{N}_\psi$. Since $\mathcal{N}_\psi\subset \Linf(\GG)$ forms a dense subset of $\Linf(\GG)$, we see that \eqref{cond1} is equivalent with $E_\omega(x) = x$. 

Using  \eqref{dense}, we can see that $y\in\tilde\sN$ if and only if 
\[(\mu\otimes \id)((\I\otimes y)\ww(P_\omega\otimes \I)) = (\mu\otimes \id)(\ww(P_\omega\otimes y))\] for all $\mu\in\Linf(\hh\GG)_*$. Equivalently $y\in\tilde\sN$ if and only if 
\[\ww^*(\I\otimes y)\ww(P_\omega\otimes \I) = P_\omega\otimes y\] which is in turn equivalent with 
\[\Delta_{\hh\GG}(y)(\I\otimes P_\omega) = y\otimes P_\omega.\] 
Since  $P_\omega\in\tilde{\sN}$ we get 
 $\Delta_{\hh\GG}(P_\omega)(\I\otimes P_\omega) = P_\omega\otimes P_\omega$. 
 
Using Podle\'s condition  $\tilde{\sN} = \{(\mu\otimes\id)(\Delta_{\hh\GG}(y)):y\in \tilde\sN,\mu\in\Linf(\hh\GG)_*\}^{ \textrm{\tiny{$\sigma$-weak cls}}}$ we conclude   that $P_\omega$ is a minimal central projection in $\tilde{\sN}$. 
 Indeed, for all $y\in \tilde\sN$ and $ \mu\in\Linf(\hh\GG)_*$ we have 
 \[(\mu\otimes\id)(\Delta_{\hh\GG}(y))P_\omega = \mu(y)P_\omega =P_\omega (\mu\otimes\id)(\Delta_{\hh\GG}(y)).\] Thus $\tilde{\sN}P_\omega = \mathbb{C}P_\omega$ (i.e. $P_\omega$ is minimal in $\tilde{\sN}$) and $P_\omega\in Z(\tilde{\sN})$. 
Minimality and centrality of $P_\omega\in \tilde{\sN}$ yields a unique normal character $\varepsilon_\omega:\tilde\sN\to\mathbb{C}$ such that $yP_\omega =\varepsilon_\omega(y)P_\omega$ for all $y\in \tilde{\sN}$.

Using $\Delta_{\hh\GG}\circ\sigma^{\hh\psi}_t = (\sigma^{\hh\psi}_t\otimes\tau^{\hh\GG}_{-t})\circ\Delta_{\hh\GG}$ (see \cite[Proposition 6.8]{KV}) we get
\begin{align*}
    \Delta_{\hh\GG}(\sigma^{\hh\psi}_t(P_\omega))(\I\otimes P_\omega) &= (\sigma^{\hh\psi}_t\otimes \tau^{\hh\GG}_{-t})(\Delta_{\hh\GG}(P_\omega))(\I\otimes P_\omega)\\&= (\sigma^{\hh\psi}_t\otimes \tau^{\hh\GG}_{-t})(\Delta_{\hh\GG}(P_\omega)(\I\otimes P_\omega))
    \\&= \sigma^{\hh\psi}_t(P_\omega)\otimes  P_\omega
\end{align*} and $\sigma^{\hh\psi}_t(P_\omega)\in\tilde\sN$. In particular $P_\omega   \sigma^{\hh\psi}_t(P_\omega) = \varepsilon_\omega(\sigma^{\hh\psi}_t(P_\omega))P_\omega$, where $\varepsilon_\omega(\sigma^{\hh\psi}_t(P_\omega))\in\{0,1\}$ for all $t\in\mathbb{R}$. Since  the map $\mathbb{R}\ni t\mapsto\varepsilon_\omega(\sigma^{\hh\psi}_t(P_\omega))\in\mathbb{R}$  is continuous and    $\varepsilon_\omega(\sigma^{\hh\psi}_t(P_\omega))|_{t=0} = 1$, we conclude that $P_\omega   \sigma^{\hh\psi}_t(P_\omega) = P_\omega$, i.e. 
$\sigma^{\hh\psi}_t(P_\omega)\geq  P_\omega$ for all $t\in\mathbb{R}$.  Thus also $\sigma^{\hh\psi}_{-t}(P_\omega)\leq  P_\omega$ for all $t\in\mathbb{R}$  and $\sigma^{\hh\psi}_t(P_\omega)= P_\omega$. 

Since $P_\omega$ is preserved by $R^{\hh\GG}$,  the identity $\Delta_{\hh\GG}(P_\omega)(\I\otimes P_\omega) =  P_\omega\otimes P_\omega$ implies that $$\Delta_{\hh\GG}(P_\omega)(P_\omega\otimes \I) = P_\omega\otimes P_\omega.$$

Finally using $\sigma^{\hh\varphi}_t = R^{\hh\GG}\circ \sigma^{\hh\psi}_{-t}\circ R^{\hh\GG}$ we get $\sigma^{\hh\varphi}_t(P_\omega)= P_\omega$ for all $t\in\mathbb{R}$. 
\end{proof}
For the concept of $\hh\GG$-generation used in the next Lemma, see Definition \ref{Ggen}.
\begin{lemma}\label{genset}
Let $\omega\in\C_0^u(\GG)^*$ be an idempotent state, $\sN = E_\omega(\Linf(\GG))$ the corresponding coideal and $\tilde\sN\subset\Linf(\hh\GG)$ the codual of $\sN$. Then  $\tilde\sN$ is $\hh\GG$-generated by $P_\omega\in\tilde\sN$. 
\end{lemma}
\begin{proof}
Let us recall that $x\in\sN$ if and only if $x\in\Linf(\GG)$ and $xP_\omega = P_\omega x$. Let $\hh V = (J\otimes J)\ww^*(J\otimes J)\in\Linf(\hh\GG)\vtens\Linf(\GG)'$ where $J:\Ltwo(\GG)\to\Ltwo(\GG)$ is the Tomita-Takesaki antiunitary conjugation assigned to $\psi$. Then for all $y\in\Linf(\hh\GG)$ we have 
\[\Delta_{\hh\GG}(y) = \hh V^*(\I\otimes y)\hh V.\] In particular if $x\in\Linf(\GG)$ and $P_\omega x = x P_\omega$ then \begin{equation}\label{eqPom}\Delta_{\hh\GG}(P_\omega)(\I\otimes x) =\hh V^*(\I\otimes P_\omega)\hh V(\I\otimes x)= (\I\otimes x)\Delta_{\hh\GG}(P_\omega).\end{equation} Conversely, if \eqref{eqPom} holds then 
\[P_\omega\otimes P_\omega x = \Delta_{\hh\GG}(P_\omega)(\I\otimes x)(P_\omega\otimes\I) = (\I\otimes x)\Delta_{\hh\GG}(P_\omega)(P_\omega\otimes\I) = P_\omega\otimes xP_\omega\] and we get $P_\omega x = x P_\omega$. In particular $\sN = \sS'\cap\Linf(\GG)$, where 
\[\sS =\{(\mu\otimes\id)( \Delta_{\hh\GG}(P_\omega)):\mu\in\Linf(\hh\GG)_*\}.\]

Let us note that $\sS''$  is the smallest  coideal of $\Linf(\hh\GG)$ containing $P_\omega$ (see Theorem \ref{thm_gen_el}). Since $\sN =\sS' \cap\Linf(\GG) =  (\sS'')'\cap\Linf(\GG)$ we get   $\sS'' = \tilde{\sN}$. 
\end{proof}

\begin{lemma}\label{corstrong}
Adopting the above notation we have  $\tau^{\hh\GG}_t(x) = \sigma^{\hh\varphi}_{t}(x)$ for all $x\in\tilde\sN$ and $t\in\mathbb{R}$.
\end{lemma}
\begin{proof}
Using the formula $\Delta^{\hh\GG}\circ\sigma^{\hh\varphi}_t =(\tau_t^{\hh\GG}\otimes\sigma^{\hh\varphi}_t) \circ \Delta^{\hh\GG}$ (see \cite[Proposition 5.38]{KV}) and 
$\Delta^{\hh\GG}\circ\tau_t^{\hh\GG} =(\tau_t^{\hh\GG}\otimes\tau_t^{\hh\GG}) \circ\Delta^{\hh\GG}$ (see \cite[Result 5.12]{KV}), we conclude that for all $\mu\in\Linf(\hh\GG)_*$ 
\[\tau^{\hh\GG}_t((\mu\otimes\id)(\Delta_{\hh\GG}(P_\omega))) = \sigma^{\hh\varphi}_{t}((\mu\otimes\id)(\Delta_{\hh\GG}(P_\omega)))\] 
(note that for the latter we also use $\tau^{\hh\GG}$-invariance and $\sigma^{\hh\varphi}$-invariance of $P_\omega$). Since $\tilde\sN$ is $\hh\GG$-generated by $P_\omega$, we are done. 
\end{proof}
Next result is a strengthening of Lemma \ref{genset}.

\begin{theorem}\label{thmstrong}
Adopting the assumptions and notation of Lemma \ref{genset} we have 
\begin{equation}\label{strong}\tilde{\sN} = \overline{\{(\mu\otimes\id)(\Delta_{\hh\GG}(P_\omega)):\mu\in\Linf(\hh\GG)_*\}}^{\text{weak}}.\end{equation}
\end{theorem}
\begin{proof}From  $\tau^{\hh\GG}$-invariance of $\tilde\sN$ it follows that   $\tilde\sN\cap D(S_{\hh\GG}^{-1})$ is a dense subset of $\tilde\sN$.
Suppose that $x\in\tilde\sN\cap D(S_{\hh\GG}^{-1})$. We shall prove that 
\begin{equation}\label{strong1}\Delta_{\hh\GG}(P_\omega)(\I\otimes x) = \Delta_{\hh\GG}(P_\omega)(S_{\hh\GG}^{-1}(x)\otimes\I).\end{equation}
From this, it follows that $\overline{\{(\mu\otimes\id)(\Delta_{\hh\GG}(P_\omega)):\mu\in\Linf(\hh\GG)_*\}}^{\text{weak}}$ is an ideal in $\tilde{\sN}$ (in particular a von Neumann subalgebra of $\tilde{\sN}$). It is also easy to check that the right hand side of Eq. \eqref{strong} is  $\hh\GG$-invariant. By ergodicity of the action of $\hh\GG$ on $\tilde\sN$, we conclude that Eq. \eqref{strong} holds (here we use the same argument as in the final part of the proof of \cite[Theorem 3.3]{KKS}). It remains to prove Eq. \eqref{strong1}. 
To this end, we continue assuming that $x$ is $\tau^{\hh\GG}$ - analytic. Note that by Corollary \ref{corstrong}, it is also $\sigma^{\hh\varphi}$ - analytic. Let $a,b\in\mathcal{N}_{\hh\varphi}$. We compute 
\begin{align*}
(\id\otimes\hh\varphi)((\I\otimes a^*)\Delta_{\hh\GG}(bP_\omega)(S_{\hh\GG}(x)\otimes\I))&=S_{\hh\GG}((\id\otimes\hh\varphi)((x\otimes \I)\Delta_{\hh\GG}(a^*)(\I\otimes bP_\omega))\\
&=S_{\hh\GG}((\id\otimes\hh\varphi)((x\otimes P_\omega)\Delta_{\hh\GG}(a^*)(\I\otimes b))\\
&=S_{\hh\GG}((\id\otimes\hh\varphi)((\I\otimes P_\omega)\Delta_{\hh\GG}(xa^*)(\I\otimes b))\\
&=S_{\hh\GG}((\id\otimes\hh\varphi)(\Delta_{\hh\GG}(xa^*)(\I\otimes bP_\omega))\\
&=(\id\otimes\hh\varphi)((\I\otimes xa^*)\Delta_{\hh\GG}(bP_\omega))\\
&=(\id\otimes\hh\varphi)((\I\otimes a^*)\Delta_{\hh\GG}(bP_\omega)(\I\otimes \sigma^{\hh\varphi}_{-i}(x)))
\end{align*} where in the first and the fifth equality, we use Eq. \eqref{stronleftinv} and in the second and the fourth equality, we use $\sigma^{\hh\varphi}$-invariance of $P_\omega$. 
Thus we get 
\[\Delta_{\hh\GG}(P_\omega)(S_{\hh\GG}(x)\otimes\I)=\Delta_{\hh\GG}(P_\omega)(\I\otimes \sigma^{\hh\varphi}_{-i}(x)).\]
Replacing $x$ with $\sigma^{\hh\varphi}_{i}(x)$ and using Corollary \ref{corstrong}, we get \eqref{strong1} for $\tau^{\hh\GG}$ - analytic $x$. Since the space of $\tau^{\hh\GG}$-analytic elements forms a core of $S^{-1}_{\hh\GG}$ we get \eqref{strong1}. 
\end{proof}
Theorem \ref{thmstrong} is a generalization of \cite[Theorem 3.3]{KKS}. 
Note that in the proof of \cite[Theorem 3.3]{KKS}, which treats the case of central $P_\omega$,  a small mistake was done where instead of Eq. \eqref{strong1} the following formula was derived: 
\[\Delta_{\hh\GG}(P_\omega)(\I\otimes x) = \Delta_{\hh\GG}(P_\omega)(R_{\hh\GG}(x)\otimes\I).\] 

The next theorem was first  proved in \cite[Theorem 5.15]{coid_sub_st}. The previous proof  strongly uses the universal $\C^*$-version of $\GG$. In what follows we give a  simpler proof   which is based on  the von Neumann version of $\GG$.  
\begin{theorem}\label{thm_normal_coid}
Let $\sN\subset \Linf(\GG)$ be an integrable normal coideal preserved by $\tau^\GG$. Then there exists a unique compact quantum subgroup $\HH\subset \GG$ such that $\sN = \Linf(\GG/\HH)$.
\end{theorem}
\begin{proof}
Using \cite[Theorem 4.2]{coid_sub_st} we conclude the existence of an idempotent state $\omega\in\C_0^u(\GG)^*$ such that  $\sN = E_\omega(\Linf(\GG))$.
Let $\tilde\sN$ be the codual coideal. Then, since $\sN$ is preserved by $\tau^\GG$, $\tilde\sN$ is preserved by $\tau^{\hh\GG}$ (see \cite[Proposition 3.2]{gpext}). Normality of $\sN$ is equivalent with $\Delta_{\hh\GG}(\tilde\sN)\subset\tilde\sN\vtens\tilde\sN$ (see \cite[Proposition 4.3]{gpext}). Moreover, using Theorem \ref{thmstrong} we see that    \[\sS = \{(\mu\otimes\id)(\Delta_{\hh\GG}(P_\omega)):\mu\in\Linf(\hh\GG)_*\}\] is weakly dense in $\tilde\sN$. 
Let us note that $R^{\hh\GG}((\mu\otimes\id)(\Delta_{\hh\GG}(P_\omega)) = (\id\otimes \mu\circ R^{\hh\GG})(\Delta_{\hh\GG}(P_\omega))$. Since $P_\omega\in\tilde\sN$ we have    $\Delta_{\hh\GG}(P_\omega)\in\tilde\sN\vtens\tilde\sN$ and we see that $R^{\hh\GG}(\sS)\subset\tilde\sN$. Thus we conclude that $R^{\hh\GG}(\tilde\sN)\subset\tilde\sN$. Summarizing $\tilde\sN$ forms a Baaj-Vaes subalgebra of $\Linf(\hh\GG)$ and there exists $\HH\subset\GG$ such that $\tilde\sN = \Linf(\hh\HH)$. Since $\sN = \Linf(\GG/\HH)$ is integrable, we  use \cite[Theorem A.3]{int} for concluding that $\HH$ is a compact quantum group. 
\end{proof}
\section{From group-like projections to idempotent states}\label{intro}
Let $\psi$ be an n.s.f. weight on a von Neumann algebra $\sN$ and $\sigma:\mathbb{R}\to\Aut(\sN)$  the KMS-group of automorphisms assigned to $\psi$. 
We denote 
\[\mathcal{T}_\psi = \{x\in\mathcal{N}_\psi\cap\mathcal{N}_\psi^*:x \textrm{ is } \sigma-\textrm{analytic and }  \sigma_z(x)\in\mathcal{N}_\psi\cap\mathcal{N}_\psi^* \textrm{ for all } z\in\mathbb{C}\}.\]
Note that if $x\in \mathcal{T}_\psi$ then $\sigma_z(x)\in\mathcal{T}_\psi$ for all $z\in \mathbb{C}$.  
Let us recall that  the KMS-condition for $\sigma$ yields that if  $x\in\mathcal{N}_\psi\cap \Dom(\sigma_{\frac{i}{2}})$ then $\sigma_{\frac{i}{2}}(x)^*\in\mathcal{N}_\psi$ and
\begin{equation}\label{kms_eq}\psi(x^*x) = \psi(\sigma_{\frac{i}{2}}(x)\sigma_{\frac{i}{2}}(x)^*).\end{equation}
\begin{lemma}\label{lem1}
Let $x\in\mathcal{T}_\psi$ and suppose that $y$ is $\sigma$-analytic. Then $yx\in\mathcal{T}_\psi$
\end{lemma}
\begin{proof} Let $x\in\mathcal{T}_\psi$. 
Clearly $yx$ is  $\sigma$-analytic. Since $\mathcal{N}_\psi$ forms a left ideal in $\sN$ we have  $yx\in \mathcal{N}_\psi$. 
Moreover $(yx)^*$ is also $\sigma$-analytic and 
\begin{align*} \psi((yx)^{**}(yx)^*)&=
\psi(\sigma_{\frac{i}{2}}((yx)^*)\sigma_{\frac{i}{2}}((yx)^*)^*)
\\
&=\psi(\sigma_{-\frac{i}{2}}(yx)^*\sigma_{-\frac{i}{2}}(yx))\\
&=\psi(\sigma_{-\frac{i}{2}}(x)^*\sigma_{-\frac{i}{2}}(y)^*\sigma_{-\frac{i}{2}}(y)\sigma_{-\frac{i}{2}}(x))
\\&\leq \|\sigma_{-\frac{i}{2}}(y)\|^2\psi(\sigma_{-\frac{i}{2}}(x)^*\sigma_{-\frac{i}{2}}(x))<\infty.
\end{align*}
Thus  we get  $yx\in\mathcal{N}_\psi\cap\mathcal{N}_\psi^*$. Replacing $x$ with $\sigma_z(x)$ and $y$ with $\sigma_z(y)$ in the above reasoning, we conclude that $\sigma_z(yx)\in\mathcal{N}_\psi\cap\mathcal{N}_\psi^*$.  Thus $yx\in\mathcal{T}_\psi$ and we are done. 
\end{proof}
\begin{remark}\label{rem1}
Let $\GG$ be a locally compact quantum group. 
In the course of the proof of the next theorem, the symbol $\hat\eta$ denotes the GNS map for the Haar weight $\hat\psi$ on $\hat\GG$. We will use the fact that if $a,b\in\mathcal{T}_{\hat\psi}$ then the slice   $(\mu_{\hat\eta(a),\hat\eta(b)}\otimes \id)(\ww)$ is an element of $\mathcal{N}_{\psi}$ (see Lemma 8.4 and Proposition 8.13   of \cite{KV} with the roles of $\GG$ and $\hat\GG$ reversed). 
\end{remark}
\begin{theorem}\label{glptoidst}
Let $\GG$ be a locally compact quantum group and let $P\in\Linf(\hh\GG)$ be a non-zero group-like projection
 such that  $\tau^{\hat\GG}_t(P) = P$ for all $t\in\mathbb{R}$.  
Then there exists an idempotent state $\omega\in\C_0^u(\GG)$ such that $P = (\id\otimes\omega)(\wW)$. 
\end{theorem}
\begin{proof}
Let us consider  $\tilde\sN\subset\Linf(\hh\GG)$, where
\[\tilde\sN = \{y\in\Linf(\hh\GG):\Delta_{\hh\GG}(y)(\I\otimes P) = y\otimes P \textrm{ and } \Delta_{\hh\GG}(y^*)(\I\otimes P) = y^*\otimes P\}.\] We will show that $\tilde{\sN}$ forms a coideal in $\Linf(\hh\GG)$ and   we will focus on its codual  $\sN\subset\Linf(\GG)$.  
Let us first note that  $P\in\tilde\sN$ and $\tilde\sN$ is $\tau^{\hh\GG}$-invariant. Moreover it is easy to see that  $\tilde\sN$ is a von Neumann subalgebra of $\Linf(\hh\GG)$. Let us  check that $\tilde\sN$ forms a coideal of $\Linf(\hh\GG)$. For  $y\in \tilde\sN$ we have
\begin{align*}
    (\id\otimes\Delta_{\hh\GG})(\Delta_{\hh\GG}(y))(\I\otimes\I\otimes P)& =  ( \Delta_{\hh\GG}\otimes\id)(\Delta_{\hh\GG}(y))(\I\otimes\I\otimes P)\\ &=( \Delta_{\hh\GG}\otimes\id)(\Delta_{\hh\GG}(y)(\I\otimes P))\\ &=( \Delta_{\hh\GG}\otimes\id)(y\otimes P) = \Delta_{\hh\GG}(y)\otimes P.
\end{align*}
Similarly we show that $   (\id\otimes\Delta_{\hh\GG})(\Delta_{\hh\GG}(y)^*)(\I\otimes\I\otimes P)= \Delta_{\hh\GG}(y)^*\otimes P$ and we get $\Delta_{\hh\GG}(y)\in\Linf(\hh\GG)\vtens\tilde{\sN}$. Repeating the reasoning presented in the fourth paragraph of the  proof of Theorem  \ref{lemma0} we conclude that  $P$ is a minimal central projection of $\tilde\sN$. Using   $\tau_t^{\hh\GG}$ invariance of $P$ and repeating the reasoning presented in the fifth  paragraph of the proof of Theorem  \ref{lemma0}, we see that $\sigma_t^{\hh\psi}(P) = P$. In particular $P$ is $ \sigma^{\hh\psi}$ - analytic. 

Let   $\sN\subset \Linf(\GG)$ denote the codual of $\tilde\sN$. Since $\tilde{\sN}$ is preserved by $\tau^{\hh\GG}$, $\sN$ is preserved by $\tau^\GG$.  
Moreover following backward the reasoning presented in the third paragraph of the proof of Theorem  \ref{lemma0} we show that   $(P\cdot\mu\otimes\id)(\ww)\in \sN$ for all $\mu\in\Linf(\hh\GG)_*$.  

Let $a,b\in\mathcal{T}_{\hat\psi}$ and  
let us consider $\mu = \mu_{\hat\eta(a),\hat\eta(b)}\in \Linf(\hh\GG)_*$ and $x = (P\cdot\mu\otimes\id)(\ww)$ (note that $P\cdot\mu = \mu_{\hat\eta(a),\hat\eta(Pb)}$). Using Lemma \ref{lem1}, we see that $Pb\in\mathcal{T}_{\hat\psi}$. In particular, as explained in Remark \ref{rem1},   $x\in\mathcal{N}_\psi$. Clearly there exists $a,b\in\mathcal{T}_{\hat\psi}$ such that the  corresponding $x$ is non-zero. Indeed, suppose the converse holds: $ (P\cdot\mu_{\hat\eta(a),\hat\eta(b)}\otimes\id)(\ww) = 0$ for all $a,b\in\mathcal{T}_{\hat\psi}$. Then $ P\cdot\mu_{\hat\eta(a),\hat\eta(b)}(y) = 0$ for all $y\in\Linf(\hh\GG)$. Thus, taking $y =\I$ we get $( \hat\eta(a)|P\hat\eta(b)) = 0$ for all $a,b\in\mathcal{T}_{\hat\psi}$. Since $\hat\eta(\mathcal{T}_{\hat\psi})$ is dense in $\Ltwo(\GG)$, we conclude that $P = 0$, contradiction. In particular $\sN$ contains a nonzero element $x\in\sN\cap\mathcal{N}_\psi$. Since $(\psi\otimes\id)\Delta_\GG(x^*x) = \psi(x^*x)$ we see that $\sN$ contains a non-zero   integrable element with respect to the action $\Delta_{\GG}|_\sN$ and using  \cite[Proposition 3.2.]{int} we conclude that $\sN$ is integrable.  

Summarizing, $\sN$ is an integrable coideal of  $\Linf(\GG)$ preserved by $\tau^\GG$. Using \cite[Theorem 4.2]{coid_sub_st} we see that there exists an idempotent state $\omega\in\C_0^u(\GG)^*$ such that  $\sN = E_\omega(\Linf(\GG))$, where $E_\omega$ is the conditional expectation assigned to $\omega$. 

Let $P_\omega = (\id\otimes\omega)(\wW)$. Then $P_\omega\in\tilde\sN$ is a minimal central projection. Moreover, 
\[(P\cdot\mu\otimes\id)(\ww) = E_\omega((P\cdot\mu\otimes\id)(\ww)) = (P_\omega P\cdot\mu\otimes\id)(\ww)\] for all $\mu\in\Linf(\hh\GG)_*$. Thus $P = P_\omega P$ and we see that $P_\omega\geq P$. Using the minimality of $P_\omega$ we get $P_\omega = P$.

\end{proof}
\section*{Acknowledgements}  RF was partially  supported by a grant from Ferdowsi University of Mashhad-Graduate studies No.37365 and  wishes to appreciate  the Department of Mathematical Methods in Physics, Faculty of Physics, University of Warsaw for their warm hospitality.
PK was partially supported by the NCN (National Center of Science) grant
 2015/17/B/ST1/00085.

\end{document}